\documentclass[11pt,a4paper]{article}
\usepackage[a4paper,
            bindingoffset=0.2in,
            left=1in,
            right=1in,
            top=1in,
            bottom=1in,
            footskip=.25in]{geometry}
\usepackage{amsmath}
\usepackage{amsfonts}
\usepackage{graphicx}
\usepackage{cite}
\usepackage{amsthm}
\usepackage{epstopdf}
\usepackage{float}
\usepackage{authblk}

\newtheorem{lemma}{Lemma}
\newtheorem{theorem}{Theorem}
\newtheorem{example}{Example}

\title{Regularization for the Approximation of 2D Set of Points via the Length of the Curve}
\author[1]{M.~E. Abbasov\thanks{m.abbasov@spbu.ru}}
\author[1]{A.~I. Belenok}
\affil[1]{St. Petersburg State University, SPbSU, 7/9 Universitetskaya nab., St. Petersburg, 199034 Russia}

\begin{document}

\maketitle

\begin{abstract}
We study the problem of approximation of 2D set of points. Such type of problems always occur in physical experiments, econometrics, data analysis and other areas. The often problems of outliers or spikes usually make researchers to apply regularization techniques, such as Lasso, Ridge or Elastic Net. These approaches always employ penalty coefficient. So the important question of evaluation of the upper bound for the coefficient arises. In the current study we propose a novel way of regularization and derive the upper bound for the used penalty coefficient.

First the problem in a general form is stated. The solution is sought in the class of piecewise continuously differentiable functions. It is shown that the optimal solution belongs to the class of piecewise linear functions. So the problem of obtaining the piecewise linear approximation that fits 2D set of point the best is stated. We show that the optimal solution is trivial and tends to a line as penalty coefficient tends to infinity. Then the main result is stated and proved. It provides the upper bound for the penalty coefficient prior to which the optimal solution differs from the line more than some pregiven positive number. We also demonstrate the proposed ideas on numerical examples which include comparison with other regularization approaches.
\end{abstract}

\subsection*{Keywords}
Regression; regularization; piecewise linear approximation; Ridge and Lasso regularization

\section{Introduction}

Regularization is actively used in statistics and machine learning to prevent overfitting and combat multicollinearity. In the case of multicollinearity, the coefficients, which are not limited by anything in the classical least squares method, can diverge either to plus or minus infinity, which negatively affects the accuracy of the model. To eliminate such problems, researchers often use Ridge regression \cite{AM_BA_Ali95,AM_BA_Hoerl70,AM_BA_ Hoerl20,AM_BA_ Hoerl75,AM_BA_Hoerl85,AM_BA_Hoerl86} or Lasso regression \cite{AM_BA_Bow76,
AM_BA_Chen98,AM_BA_Tibshirani12,AM_BA_Zou07,AM_BA_Lee16,AM_BA_Lockhart14}. In the case of Lasso regression, instead of the residual sum of squares (RSS), it is proposed to minimize the expression $RSS+\alpha \sum^n_{i=1}{|{\beta }_i|}$, where $\beta_i$, $i=1,\dots,n$ are coefficients of the model and $\alpha$ is a scalar penalty parameter. In case of Ridge regression the expression $RSS+\alpha \sum^n_{i=1}{{{\beta }_i}^2}$ is minimized.
These approaches are usually used to treat bias-variance tradeoff. The main idea here is to decrease the variance of the model at the cost of increasing its bias. This way one can mitigate negative effects such as overfitting of the model or unwanted spikes of the obtained approximation.

RIDGE and Lasso regressions produce similar results for small values of the penalty coefficient $\alpha$, but it is interesting that Lasso regression, unlike its smooth counterpart, not only performs regularization, but also resets some coefficients to zero for a sufficiently large penalty value. That is, the Lasso approach reduces the number of variables, which makes it easier to interpret the model \cite{AM_BA_BGP02,AM_BA_Con96}. Another significant difference between these models is that $l_1$-regularization, which is employed by Lasso regression, leads to the a slower reduction of penalty $\alpha \sum^n_{i=1}{|{\beta }_i|}\ $. This is important for the most valuable variables in the model, and therefore significant for obtaining more accurate forecasts \cite{AM_BA_Coo03}.

RIDGE regression shows higher sensitivity to outliers than its Lasso mate. 
While the Lasso regression has some advantages over Ridge regression, Lasso regression also has a number of disadvantages: the Lasso estimate does not guarantee uniqueness in all situations \cite{AM_BA_Ell98}, the process of adjusting the $\alpha$ parameter tends to give unstable results \cite{AM_BA_Fle80}, in addition, it has been empirically shown that Lasso performs worse in a situation where the set of true values contains many small, but not zero, observations \cite{AM_BA_FGK03}.

To eliminate some of the problems associated with Lasso and Ridge regression, Elastic Net regularization was proposed, which is a hybrid of the methods discussed above: in this case, the expression $RSS+{\alpha }_1\sum^n_{i=1}{{\beta }_i}^2+\alpha_2\sum^n_{i=1}{|{\beta }_i|}$ is minimized \cite{AM_BA_GMW81}.

Despite the large number of works on this topic in practical application researchers often face the problem of choosing the value of the penalty coefficient. At the moment, the most commonly used technique for these purposes is cross-validation, which helps to find $\alpha $ from available observations \cite{AM_BA_GHGsoft}. The cross-validation procedure artificially divides observation data into test and training samples. Estimates of regression coefficients are based on the training sample. Then the resulting regression model is tested on the test sample. It is obvious that it is possible to obtain different penalty coefficients, depending on the method of dividing the original sample.

This article considers piecewise linear approximations of a set of two-dimensional points. A regularization method based on the length of the broken line is proposed. An upper estimate for the penalty coefficient is obtained.

The paper is organized as follows. A novel idea of regularization as well as its preliminary analysis, which is important for the rest of the research are given in Section 2. In Section 3 we present auxiliary lemmas. They are used to derive the main result regarding the upper bound of the penalty coefficient. Some numerical examples demonstrating obtained results and including comparison with the other regularization techniques are provided in Section 4.

\section{The statement of the problem}\label{AM_BA_subsect_statement}

Consider a set of points $$X=\{(x_i,y_i)\mid i=1,\dots,n\}$$ such that $x_i\neq x_j$ if $i\neq j$. In order to get an approximation of the set via a curve $y(x)$ defined by some parameters one can employ standard regression technique by minimizing $$RSS=\sum^n_{i=1}{\left(y_i-y\left(x_i\right)\right)}^2$$
subject to these parameters. The resulting curve can have spikes, say due to outliers in $X$. In order to overcome this we propose to minimize the following functional
$$J_\alpha\left(y\right)=\sum^n_{i=1}{{\left(y_i-y\left(x_i\right)\right)}^2+\alpha \int^l_0{\sqrt{1+y^{'2}\left(x\right)}\ dx}\ },$$
which is the sum of $RSS$ and the lengths of the approximation curve multiplied by penalty coefficient $\alpha$. The idea of introduction of second term is plain: spikes produce segments of the curve with high length, therefore taking into account the whole length of the curve with some multiplier we can penalize curve for spikes. The impact of penalization can be adjusted by selecting $\alpha$. Obviously, the higher $\alpha$ the more we are focusing efforts on minimizing the lengths of $y(x)$ and as $\alpha$ tends to infinity $y(x)$ tends to a line. Therefore, natural question occurs: what is the upper bound for $\alpha$ for which the problem still makes sense? We need to clarify some details before the dive into the study of the mentioned questions.

Assume that $J_\alpha\left(y\right)$ attains minimum on a function $\widetilde{y}(x)$. Consider a segment $\left[x_{i-1};\ x_i\right]$. 
The section from $x_{i-1}$ to $x_i$ of the function $\widetilde{y}(x)$ must be rectilinear, otherwise by replacing this section with a line segment connecting the points $(x_i,\ {\widetilde{y}(x}_i))$ and $(x_{i-1},\ {\widetilde{y}(x}_{i-1}))$ we reduce the value of the functional.
This implies that the function which delivers the minimum to the functional is a broken line with vertices in points $x_i$, $i=1,\dots,k$. Therefore, in what follows we work only within the class of piecewise linear approximations.

Let us find the broken line which provides the best approximation of the set of points $X$ in the sense of minimization of  functional $J_{\alpha}$. Passing to broken lines, the original functional takes the form $$J_{\alpha}\left(a\right)=\sum^n_{i=1}{{\left(y_i-a_i\right)}^2+\alpha \sum^n_{i =2}{\sqrt{{(a_i-a_{i-1})}^2+{(x_i-x_{i-1})}^2}}\ },$$ where $a=(a_1,\dots,a_n)^T$ and $a_i=y\left( x_i\right)$, $i=1,\dots,n$ are ordinates of vertices of the broken line. We need to determine $a$ at which $J_{\alpha}$ reaches a minimum. Let us use the necessary condition for a minimum:
\begin{equation}\label{AM_BA_Ja_1}
\begin{split}
\frac{{\partial J}_{\alpha }}{\partial a_i}=&2\left(y_i - a_i\right)+\alpha \Bigg[ \frac{a_i-a_{i-1}} {\sqrt{{\left(a_i-a_{i-1}\right)}^2+{\left(x_i-x_{i-1}\right)}^2}}\\ &-\frac{a_{i +1}-a_i}{\sqrt{{\left(a_{i+1}-a_i\right)}^2+{\left(x_{i+1}-x_i\right)}^2}}\Bigg]=0, \quad \forall i=2,\dots,n-1.
\end{split}
\end{equation}
For $i=1$ and $i=n$ the expressions are the same except the fact that the first or the second term in the square brackets respectively should be omitted.

It is clear from the resulting expression that for small $\alpha$ we obtain a solution which is close to $a_i=y_i$, $i=1,\dots, n,$ that is, the optimal broken line is close to the broken line with vertices $\left(x_i,\ y_i\right)$, $i=1,\dots,n$. For large values of $\alpha$, the first term in (\ref{AM_BA_Ja_1}) becomes negligible and the solution is close to a broken line, the links of which have a close slope, that is, the solution differs little from the straight line, which, as is easy to see based on the form of the functional, is parallel to axis $Ox$. Let us find out what this line is.

\section{Deriving an upper bound for the penalty coefficient}\label{AM_BA_subsect_up_bound}

The following result defines a line to which an optimal broken line tends as penalty coefficient $\alpha$ tends to infinity.

\begin{lemma}\label{AM_BA_lem1}
Let $$\Omega =\left\{a={{\left(k_1x_1+k_2,\dots,k_1x_n+k_2\right)}}\mid k_1, k_2\in \mathbb{R}\right\} .$$ Then ${\displaystyle\underset{a\in \Omega}{\operatorname{argmin}} {J_{\alpha }}(a)=\overline{a}\ },$ where $\overline{a}=\left(\overline{a}_1,\dots,\overline{a}_n\right)$, $$\overline{a}_i=\frac{1}{n}\sum^n_{j =1}{y_j},\quad\forall i=1,\dots,n.$$ In other words, on the class of lines $J_{\alpha }\left(a\right)$ reaches a minimum on a straight line parallel to the $Ox$ axis. So, this is the average line for the ordinates of all points of set $X$.
\end{lemma}

\begin{proof}

The functional $J_{\alpha }$ on $\Omega$ gets the form
$$J_{\alpha }\left(k_1,\ k_2\right)=\sum^n_{i=1}{{(k_1x_i+k_2-y_i)}^2+\alpha \sum^{n-1} _{i=1}{\sqrt{{\left(k_1x_{i+1}-k_1x_i\right)}^2+{\left(x_{i+1}-x_i\right)}^2}}},$$
whence
$$J_{\alpha }\left(k_1,\ k_2\right)=\sum^n_{i=1}{{(k_1x_i+k_2-y_i)}^2+\alpha (x_n-x_1)\sqrt{1+{k_1}^2}}.$$

Employing necessary condition for the minimum we get the system of two equations
 
\begin{equation}
\begin{cases}
\displaystyle\frac{{\partial J}_{\alpha }\left(k_1,\ k_2\right)}{\partial k_1}=2\sum^n_{i=1}{\left(k_1x_i+k_2-y_i\right)x_i+\alpha \left(x_n-x_1\right)\frac{k_1}{\sqrt{1+{k_1}^2}}=0} \\
\displaystyle\frac{{\partial J}_{\alpha }\left(k_1,\ k_2\right)}{\partial k_2}=2\sum^n_{i=1}{\left(k_1x_i+k_2-y_i\right)=0} 
\end{cases}
\end{equation}
which has the solution $k_1=0$, $k_2=\displaystyle\frac{1}{n}\sum^n_{j=1}{y_j}$.

\end{proof}

In what follows we will also need the following auxiliary result too.

\begin{lemma}\label{AM_BA_lem2}
The coordinates of the vector $a$ on which the minimum of the functional $J_{\alpha }\left(a\right)$ is acttained belongs to the segment 
\begin{equation}\label{AM_BA_lem2_cond}\displaystyle\left[{\min_{i=1,\dots,n}y_i,\max_{i=1,\dots,n}y_i}\right],\end{equation} that is 
$$
\min_{i=1,\dots,n}y_i\leq a_i\le \max_{i=1,\dots,n}y_i ,\quad i=1,\dots,n.
$$
\end{lemma}

\begin{proof}
Assume the contrary. Let there exists $k\in\left[1,\dots ,n\right]$ such that $$a_k>\max_{i=1,\dots,n}y_i.$$ Without loss of generality we can assume that $a_i$, $i=1,\dots,k-1,k+1,\dots,n$ belong to the segment \eqref{AM_BA_lem2_cond}.

In point $x_k$ the square of the deviation is equal to $(a_k-y_k)^2,$ and the links of the broken line having the point $(x_k, a_k)$ as an endpoint contribute the value $$\alpha\left(\sqrt{{\left(a_k-a_{k-1}\right)}^2+{\left(x_k-x_{k-1}\right)}^2}+\sqrt{{\left( a_{k+1}-a_k\right)}^2+{\left(x_{k+1}-x_k\right)}^2}\right)$$ to functional $J_{\alpha }$.

\begin{figure}[h]
\centering
\resizebox*{8cm}{!}{\includegraphics{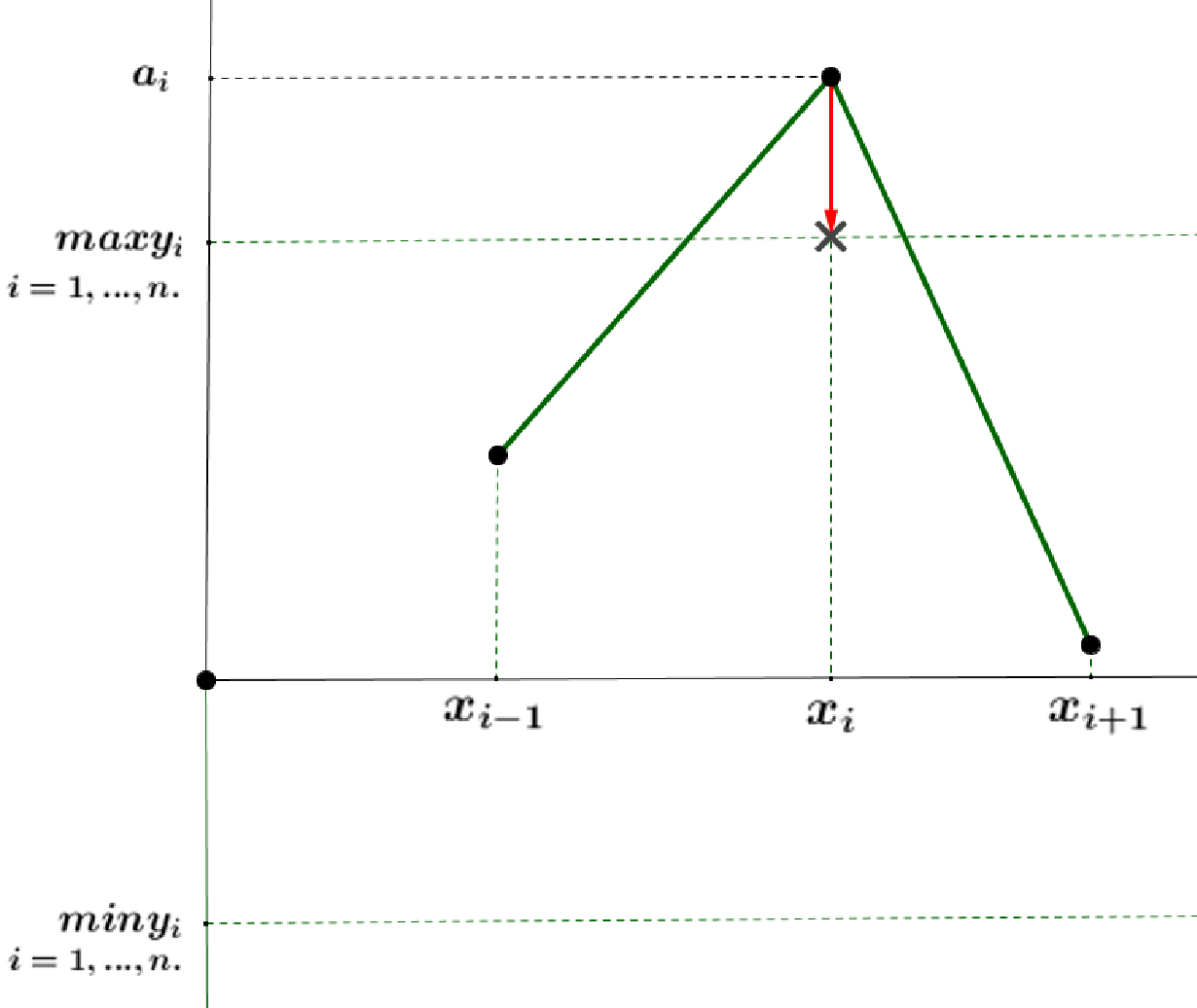}}\hspace{7pt}
\caption{Illustration to the proof of Lemma \ref{AM_BA_lem2}}\label{AM_BA_fig1_lem2}
\end{figure}

We have
$$a_k-a_{k-1}>\max_{i=1,\dots,n}y_i-a_{k-1}\geq 0$$
$$a_k-a_{k+1}>\max_{i=1,\dots,n}y_i-a_{k+1}\ge 0,\quad i=1,\dots ,n.$$ 
Therefore, replacement of $a_k$ with $\displaystyle\max_{i=1,\dots,n}y_i$ leads to a strict decrease of $J_{\alpha }.$

In addition,$$a_k-y_k>\max_{i=1,\dots,n}y_i-y_k\geq 0.$$
Therefore, replacing $a_k\ $ with $\max_{i=1,\dots,n}y_i$ also leads to a strict decrease (of the corresponding term from the sum of squared residuals) of $J_{\alpha }$.

Thus, replacing $a_k\ $ with $\displaystyle\max_{i=1,\dots,n}y_i$ we strictly reduce the value of the functional $J_{\alpha }$. Therefore, the original piecewise linear trajectory corresponding to the vector $a$ can not be the  optimal one, which is a contradiction.

The case $a_k<\displaystyle\min_{i=1,\dots,n}y_i$ can be studied similarly.
\end{proof}

Now we can proceed to the most important question here: what is the upper bound for $\alpha$ starting from which the solution of the problem differs from the average line in Lemma \ref{AM_BA_lem1} less than some small positive $\varepsilon$? 
\begin{figure}[h]
\centering
\resizebox*{8cm}{!}{\includegraphics{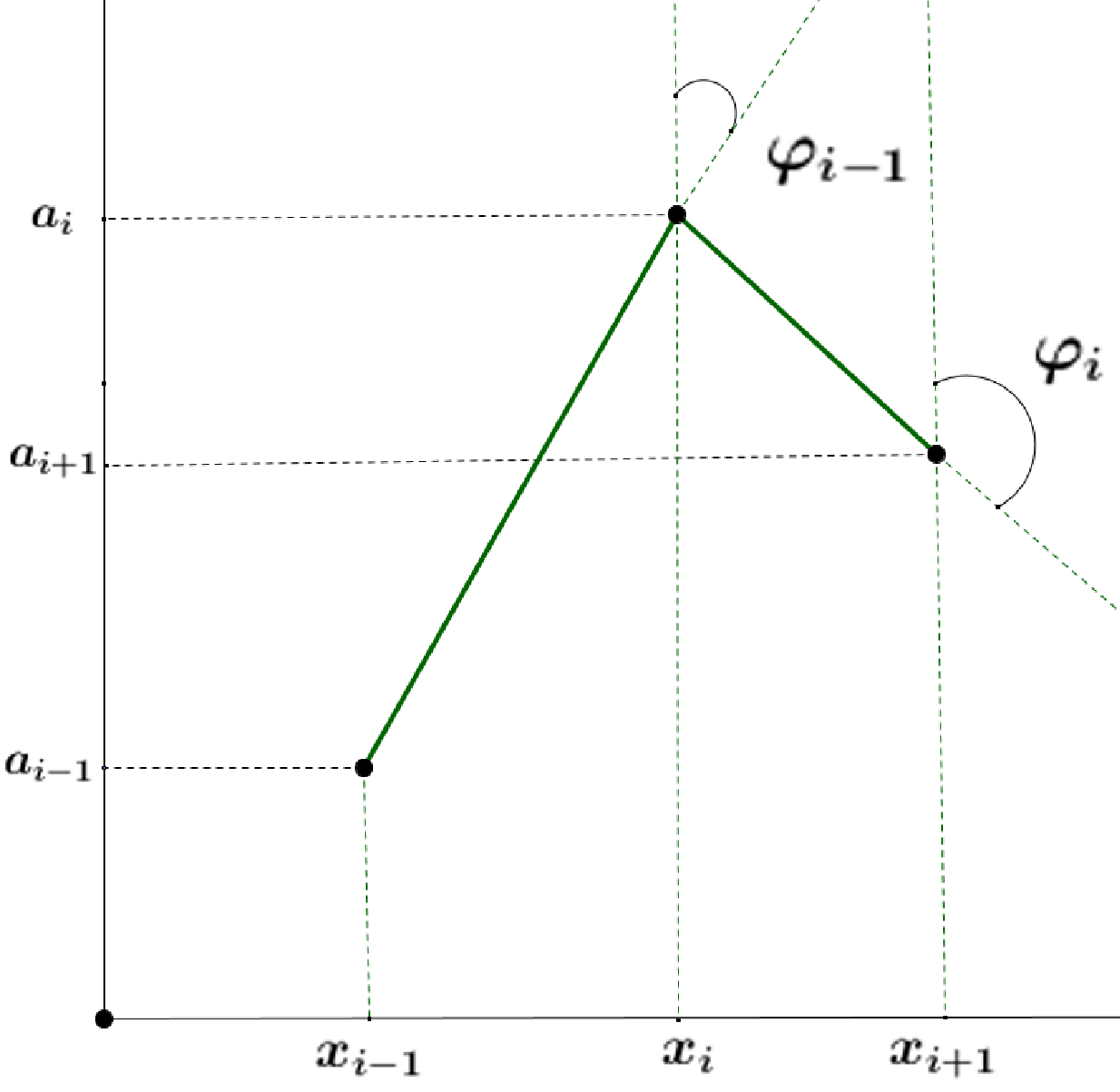}}\hspace{7pt}
\caption{The angles which characterize the slopes of the broken line links}\label{AM_BA_fig1}
\end{figure}

Let us rewrite the terms in square brackets in expression (\ref{AM_BA_Ja_1}) by introducing the angle ${\varphi }_i$, which is the angle of inclination of the $i$-th link of the broken line (i.e. a link that is a segment with ends at points $\left(x_{i},\ a_{i}\right)$ and $\left(x_{i+1},\ a_{i+1}\right)$) to the positive direction of the $Oy$ axis. 
The angle is measured so that the rotation from the broken line link to the $Oy$ axis be counterclockwise (see Fig. \ref{AM_BA_fig1}).

We have
\begin{equation}
\begin{cases}
\displaystyle\frac{{\partial J}_{\alpha }}{\partial a_1}&=2\left(a_1-y_1\right)-\alpha \cos{\varphi_1}=0,\\
\displaystyle\frac{{\partial J}_{\alpha }}{\partial a_i}&=2\left(a_i-y_i\right)+\alpha \left[\cos{\varphi_{i-1}}-\cos{\varphi_i}\right]=0,\quad \forall i=2,\dots,n-1,\\
\displaystyle\frac{{\partial J}_{\alpha }}{\partial a_n}&=2\left(a_n-y_n\right)+\alpha \cos{\varphi_n}=0.
\end{cases}
\end{equation}

So, $\varphi=(\varphi_1,\dots,\varphi_n)^T$ must be determined from the system
\begin{equation}\label{AM_BA_main_syst1}
\begin{cases}
\displaystyle\frac{2\left(a_1-y_1\right)}{\alpha }-{\cos{\varphi }_1\ }=0, \\ 
\cdots \\ 
\displaystyle\frac{2\left(a_i-y_i\right)}{\alpha }+{\cos{\varphi}_{i-1}\ }-{\cos{\varphi }_i\ }=0,\\ 
\cdots \\ 
\displaystyle\frac{2\left(a_n-y_n\right)}{\alpha }+{\cos{\varphi }_n\ }=0.
\end{cases}
\end{equation}

It must be noted that in what follows we are going to deal with maximum norm $$\|\phi\|_\infty=\max_{i=1,\dots,n}|\phi_i|.$$ Now we can proceed to the main result.

\begin{theorem} \label{AM_BA_main_th}
Let $\delta=(\delta_1,\dots,\delta_n)^T$ be a vector of angles of inclination of the broken line links to the axis $Ox$ and $\varepsilon$ be a positive number.
If the penalty coefficient satisfies the inequality
$$\alpha < \frac{\displaystyle\max_{i=1,\dots,n}y_i -\min_{i=1,\dots,n}y_i}{\varepsilon },$$
then $\phi_i=\frac{\pi}{2}+\delta_i,$ $i=1,\dots,n$ where
\begin{equation}\label{AM_BA_cor_cond_1}
\left\|\delta \right\|_{\infty}> \varepsilon,
\end{equation}
\end{theorem}

\begin{proof}
Let us introduce the the vector functions:
$$F(\varphi)={(-{\mathrm{cos} {\varphi }_1\ },\ {\mathrm{cos} {\varphi }_1\ }-{\mathrm{cos} {\varphi }_2\ },\ \dots ,\ {\mathrm{cos} {\varphi }_{n-2}\ }-{\mathrm{cos} {\varphi }_{n-1},\ {\mathrm{cos} {\varphi }_n\ })\ }}^{{T}},$$
$$c_\alpha(a,y)=\left(\frac{2\left(a_1-y_1\right)}{\alpha},\dots,\frac{2\left(a_n-y_n\right)}{\alpha}\right)^T$$
Denote $\omega ={(1,\dots ,1)}^{{T}}$. System \eqref{AM_BA_main_syst1} can be rewritten in a vector form as
$$F\left(\varphi\right)+c_\alpha(a,y)=0_n$$

Obviously, when the broken line degenerates into the average line we have 
$$\varphi=\frac{\pi }{2}\omega,$$ and $$F\left(\frac{\pi }{2}\omega \right)=0_n.$$

Consider the equation 
\begin{equation}\label{AM_BA_main_eq1}
F\left(\frac{\pi }{2}\omega +\delta \right)+c_\alpha(a,y)=0_n,
\end{equation}
which determines the behavior of the solution in the neighborhood of the average line.
Here $\delta\in\mathbb{R}^n$ is a vector of angle deviations. Let us expand $F$ in a neighborhood of the point $\delta=0_n$.
$$F\left(\frac{\pi }{2}\omega +\delta \right)=F\left(\frac{\pi }{2}\omega \right)+F'\left(\frac {\pi }{2}\omega +\theta \delta \right)\delta,$$
where $F'$ is the Jacobian of the function $F$, and $\theta \in\left[0,1\right].$

Thus, the equation \eqref{AM_BA_main_eq1} can be rewritten as:
\begin{equation*}\label{AM_BA_2}
F'\left(\frac{\pi }{2}\omega +\theta \delta \right)\delta +c_\alpha(a,y)=0,
\end{equation*}
whence
\begin{equation}\label{AM_BA_3}
\left\|c_\alpha(a,y) \right\|_{\infty}\leq\left\|{F'\left(\frac{\pi }{2}\omega +\theta \delta \right)}\right\|_{\infty}\left\|\delta\right\|_{\infty}.
\end{equation}

Let us take a closer look at Jacobian 
$$F'\left(\varphi \right)=
\begin{pmatrix}
\sin{\varphi_1} & 0 &0 & \cdots & 0&  0& 0  \\
-\sin{\varphi_1} & \sin{\varphi_2} &0 &  \cdots&0& 0 & 0\\
0 & -\sin{\varphi_2} & \sin{\varphi_3}  & \cdots&0& 0 & 0 \\
\vdots  & \vdots & \vdots & \ddots &  \vdots&\vdots& \vdots  \\
0 & 0& 0&\cdots& -\sin{\varphi_{n-2}} & \sin{\varphi_{n-1}} & 0 \\
0 & 0 & 0&\cdots&  0 & 0 & -\sin{\varphi_n}
\end{pmatrix}_{n\times n}.$$
As we see $F'(\varphi )$ is a lower bidiagonal matrix.

It is easy to check that $$\|F'(\varphi )\|_{\infty}\leq 2.$$ Therefore, from \eqref{AM_BA_3} we have that
the inequality $\|\delta\|_{\infty}\leq\varepsilon$ implies $$\|c_\alpha(a,y) \|\leq 2\varepsilon,$$ whence 
if \begin{equation}\label{AM_BA_4}
\|c_\alpha(a,y) \|_{\infty}> 2\varepsilon
\end{equation}
then $\|\delta\|_{\infty}>\varepsilon$.
But 
$$\|c_\alpha(a,y) \|_{\infty}=\frac{2\displaystyle\max_{i=1,\dots,n}|a_i-y_i|}{\alpha}$$

Therefore, \eqref{AM_BA_4} brings us to inequality
$$\alpha<\frac{\displaystyle\max_{i=1,\dots,n}|a_i-y_i|}{\varepsilon}.$$
 right side of which can be bounded from above by means of Lemma \ref{AM_BA_lem2}. Finally we get
$$\alpha <\overline{\alpha}= \frac{\displaystyle\max_{i=1,\dots,n}y_i -\min_{i=1,\dots,n}y_i}{\varepsilon }.$$ 
\end{proof}

It must be noted that inequality \eqref{AM_BA_cor_cond_1} means that the optimal broken line differs not less than $\varepsilon$ from the line in Lemma \ref{AM_BA_lem1}. Consequently, Theorem \ref{AM_BA_main_th} provides the upper bound for the $\alpha$ for which we can guarantee that the solution is not close to the trivial one.

\section{Results of numerical experiments}

\begin{example}\label{AM_BA_exmpl1}
Let us consider a set $X$ which consists of $12$ points:
\begin{equation*}
\begin{split}
X=\bigg[&\left(0,8.74970396\right); \left(\frac{1}{11},-20.44713479\right); \left(\frac{2}{11},-32.91463106\right); \left(\frac{3}{11},3.15417017\right),\\
&\left(\frac{4}{11},7.00412727\right); \left(\frac{5}{11},28.91840718\right); \left(\frac{6}{11},30.48287838\right); \left(\frac{7}{11},3.75438939\right);\\
&\left(\frac{8}{11},-0.06462388\right); \left(\frac{9}{11},0.46098988\right); \left(\frac{10}{11},-12.35515561\right); \left(1,-40.121391\right)\bigg].
\end{split}
\end{equation*}

Choose $$\varepsilon = 0.105.$$ Then according to Theorem \ref{AM_BA_main_th} the upper bound for penalty coefficient $\overline{\alpha}=672.4$. 

\begin{figure}[H]
\centering
\resizebox*{10cm}{!}{\includegraphics{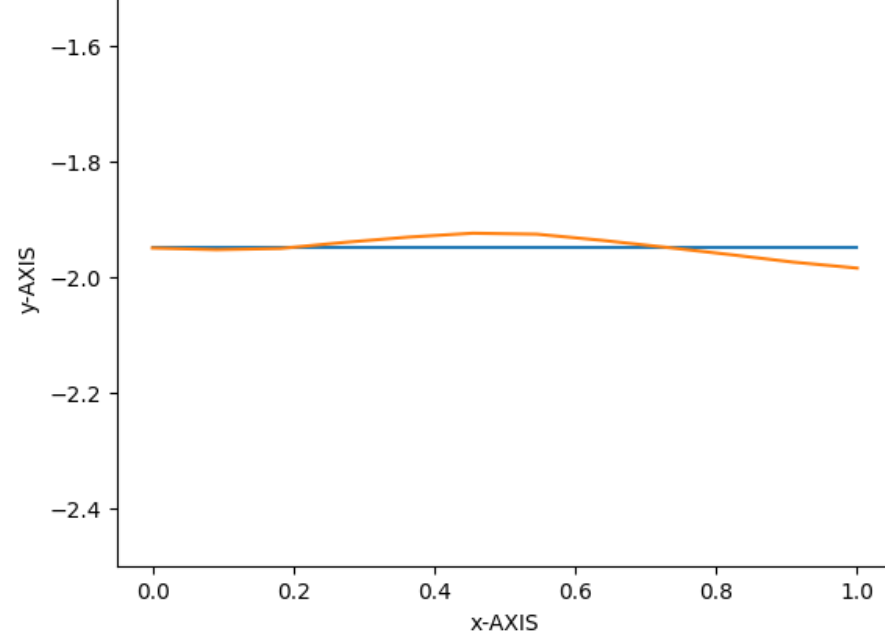}}\hspace{0pt}
\caption{The optimal broken line (orange) and the averege line (blue) in Example \ref{AM_BA_exmpl1}}
\label{AM_BA_exmpl1_fig}
\end{figure}

Functional $J_{\overline{\alpha}}$ attains the minimum on the broken line (see Fig. \ref{AM_BA_exmpl1_fig}) which is defined by vector of ordinates of its vertices $a$ 
\begin{equation*}
\begin{split}
a=(&-1.94981404, -1.95270849, -1.95060026, -1.9400497,\\ 
&-1.93090088, -1.92419569, -1.92584876, -1.93633343,\\
& -1.94839252, -1.96097549, -1.97423314, -1.98461195)^T.
\end{split}
\end{equation*}

\end{example}

\begin{example}\label{AM_BA_exmpl2}
Let us consider a set $X$ which consists of $12$ points:
\begin{equation*}
\begin{split}
X=\bigg[&\left(0,2\right); \left(\frac{1}{11},4\right); \left(\frac{2}{11},8\right); \left(\frac{3}{11},1\right),\left(\frac{4}{11},0.5\right); \left(\frac{5}{11},0\right); \left(\frac{6}{11},25\right); \\
&\left(\frac{7}{11},10\right); \left(\frac{8}{11},0\right); \left(\frac{9}{11},125\right); \left(\frac{10}{11},14\right); \left(1,12\right)\bigg].
\end{split}
\end{equation*}

Choose $$\varepsilon =0.1.$$ Then according to Theorem \ref{AM_BA_main_th} the upper bound for penalty coefficient $\overline{\alpha}=1250$.

\begin{figure}[h]
\begin{minipage}[h]{0.49\linewidth}
\center{\includegraphics[width=1\linewidth]{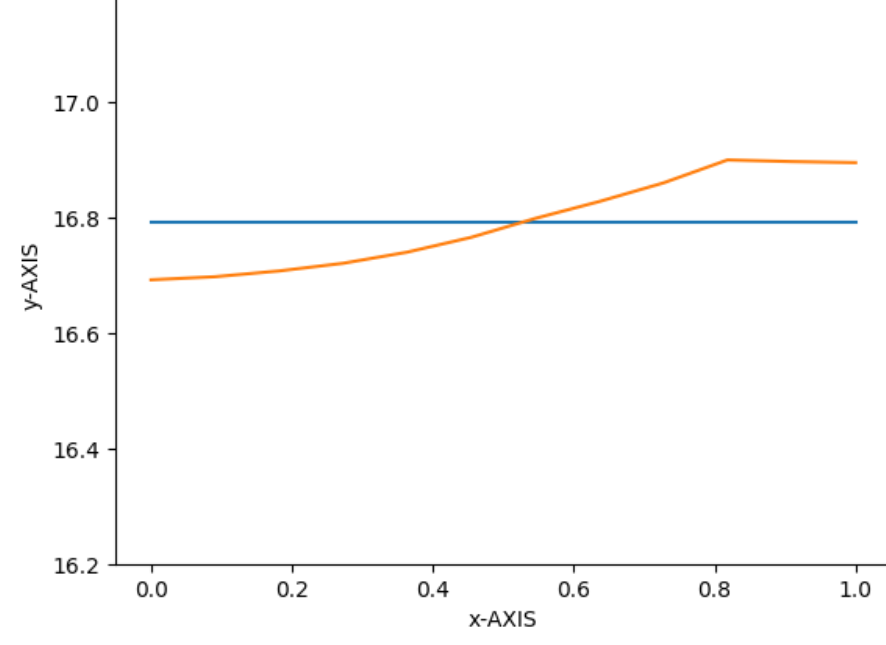}} \\ a) 
\end{minipage}
\hfill
\begin{minipage}[h]{0.49\linewidth}
\center{\includegraphics[width=1\linewidth]{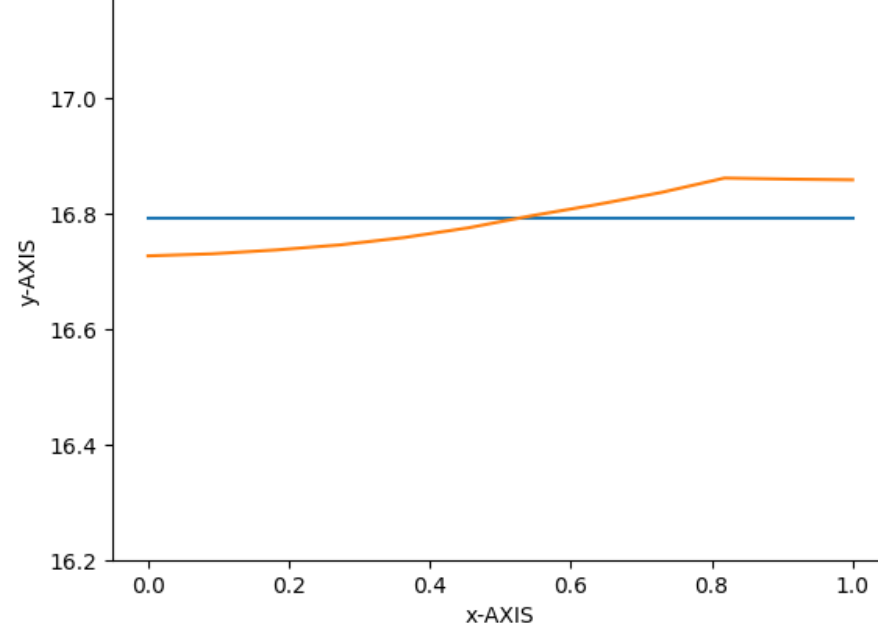}} \\ b)
\end{minipage}
\vfill
\begin{minipage}[h]{0.49\linewidth}
\center{\includegraphics[width=1\linewidth]{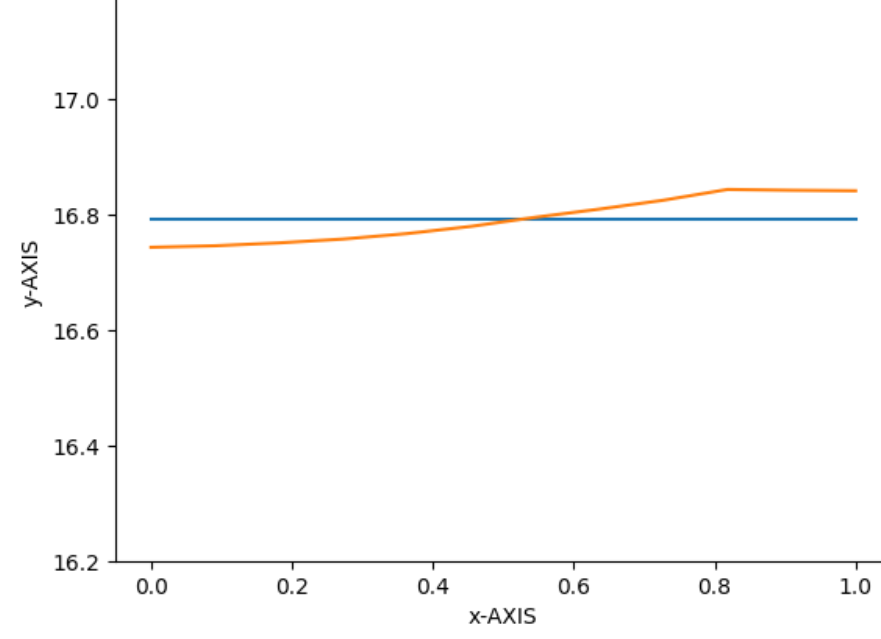}} \\ c)
\end{minipage}
\hfill
\begin{minipage}[h]{0.49\linewidth}
\center{\includegraphics[width=1\linewidth]{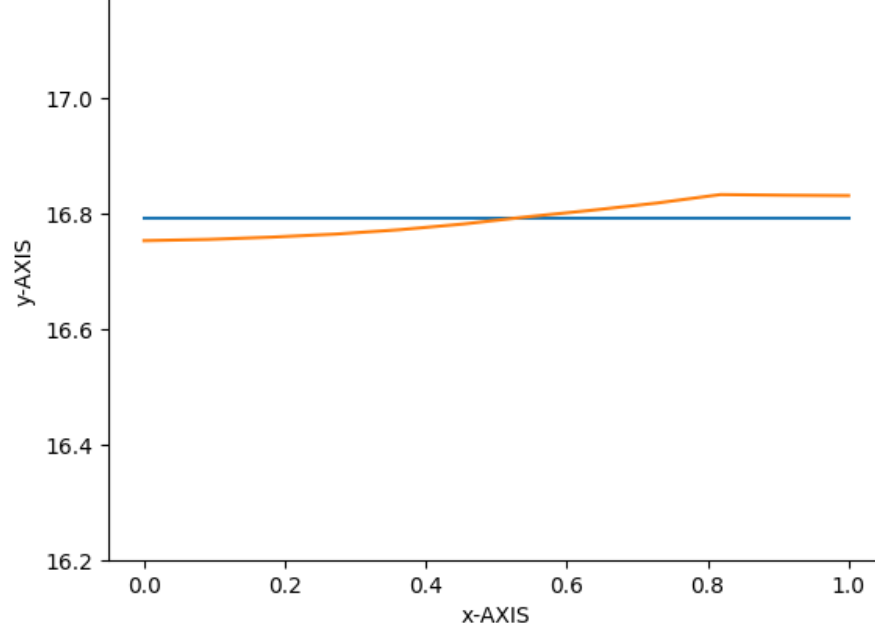}} \\ d)
\end{minipage}
\caption{The optimal broken line (orange) with $\overline{\alpha}$ values 500 (Fig. a), 750 (Fig. b), 1000 (Fig. c), 1250 (Fig. d) and the average line (blue) in Example \ref{AM_BA_exmpl2}}
\label{AM_BA_exmpl1_fig2}
\end{figure} 

\newpage

Fig.\ref{AM_BA_exmpl1_fig2} clearly shows that as $\overline{\alpha}$ increases, the deviation of the curve obtained using the LSM (least squares method) with proposed regularization and the straight line, that passes through the average value of $Y$, decreases.

\end{example}

\begin{example}\label{AM_BA_exmpl3}

Let us consider a set $X$ which consists of $12$ points:
\begin{equation*}
\begin{split}
X=\bigg[&\left(0,0\right); \left(1,20\right); \left(2,15\right); \left(3,5\right),\left(4,12\right); \left(5,5\right); \left(6,6\right); \\
&\left(7,5\right); \left(8,5\right); \left(9,7\right); \left(10,19\right); \left(11,2\right)\bigg].
\end{split}
\end{equation*}

In this example we will compare proposed approach of regularization with its Ridge and Lasso competitors. In the case of Ridge regression with piecewise linear approximation the squares of the angular coefficient in each section $\left[x_{i-1};\ x_i\right]$ are summed as a penalty function.When using Lasso regression, the sum of the modules of the angular coefficients is considered as a penalty function.

Here and later on we use iteratively reweighted least squares (IRLS) \cite{AM_BA_IRLS, AM_BA_IRLS1, AM_BA_IRLS2, AM_BA_IRLS3} method to minimize the function obtained by means of Lasso regression.

\begin{figure}[h]
\begin{minipage}[h]{0.49\linewidth}
\center{\includegraphics[width=1\linewidth]{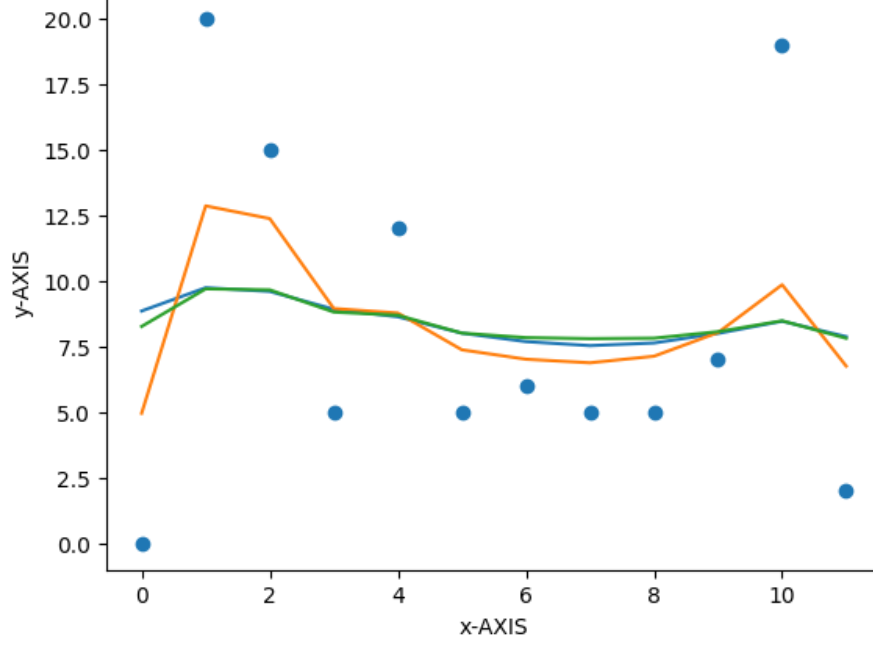}} \\ a)
\end{minipage}
\hfill
\begin{minipage}[h]{0.49\linewidth}
\center{\includegraphics[width=1\linewidth]{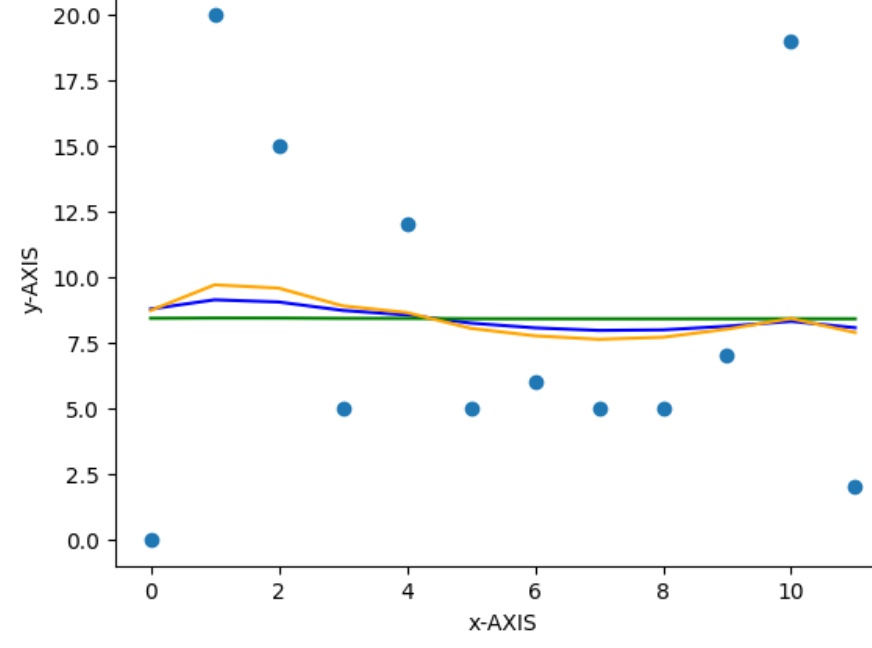}} \\ b)
\end{minipage}
\vfill
\begin{minipage}[h]{0.49\linewidth}
\center{\includegraphics[width=1\linewidth]{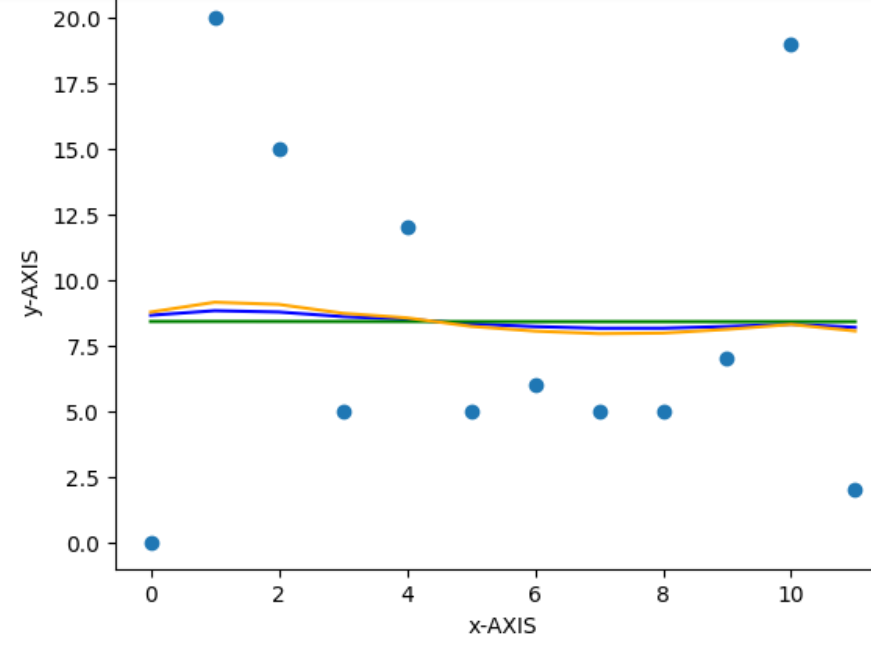}} \\ c)
\end{minipage}
\hfill
\begin{minipage}[h]{0.49\linewidth}
\center{\includegraphics[width=1\linewidth]{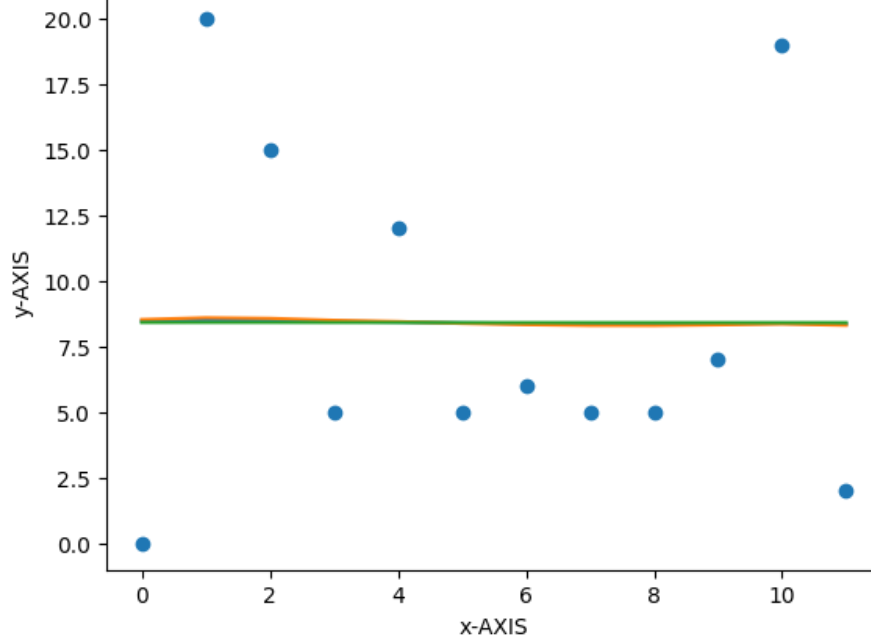}} \\ d)
\end{minipage}
\caption{The optimal broken line constructed by the proposed modification (orange), Ridge regression (blue), Lasso regression (green) at $\overline{\alpha}$ values $10$ (Fig. a), $25$ (Fig. b), $50$ (Fig. c), $250$ (Fig. d) in Example \ref{AM_BA_exmpl3}}
\label{AM_BA_exmpl1_fig3}
\end{figure} 

Fig.\ref{AM_BA_exmpl1_fig3} shows that the LSM (least squares method) with proposed regularization differs from Ridge and Lasso regression.

\end{example}

\begin{example}\label{AM_BA_exmpl4}
Let us consider a set $X$ which consists of $12$ points:
\begin{equation*}
\begin{split}
X=\bigg[&\left(0,-32\right); \left(1,2048\right); \left(2,0\right); \left(3,256\right),\left(4,-128\right); \left(5,16\right); \left(6,-4\right); \\
&\left(7,2\right); \left(8,128\right); \left(9,-64\right); \left(10,-32\right); \left(11,8\right)\bigg].
\end{split}
\end{equation*}

Choose $\varepsilon =0.05.$ Then according to Theorem \ref{AM_BA_main_th} the upper bound for penalty coefficient $\overline{\alpha}=43520$. 

\begin{figure}[h]
\begin{minipage}[h]{0.49\linewidth}
\center{\includegraphics[width=1\linewidth]{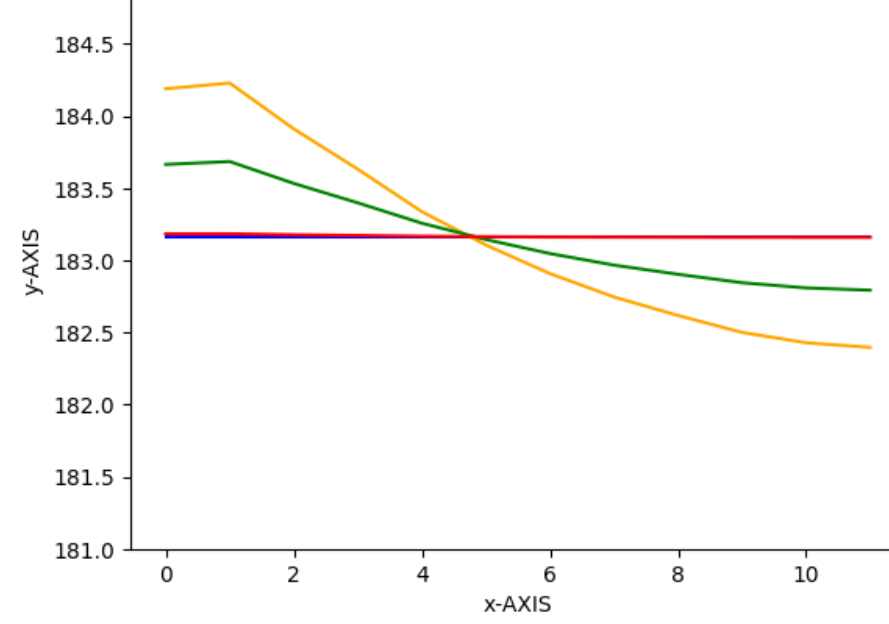}} \\ a) 
\end{minipage}
\hfill
\begin{minipage}[h]{0.49\linewidth}
\center{\includegraphics[width=1\linewidth]{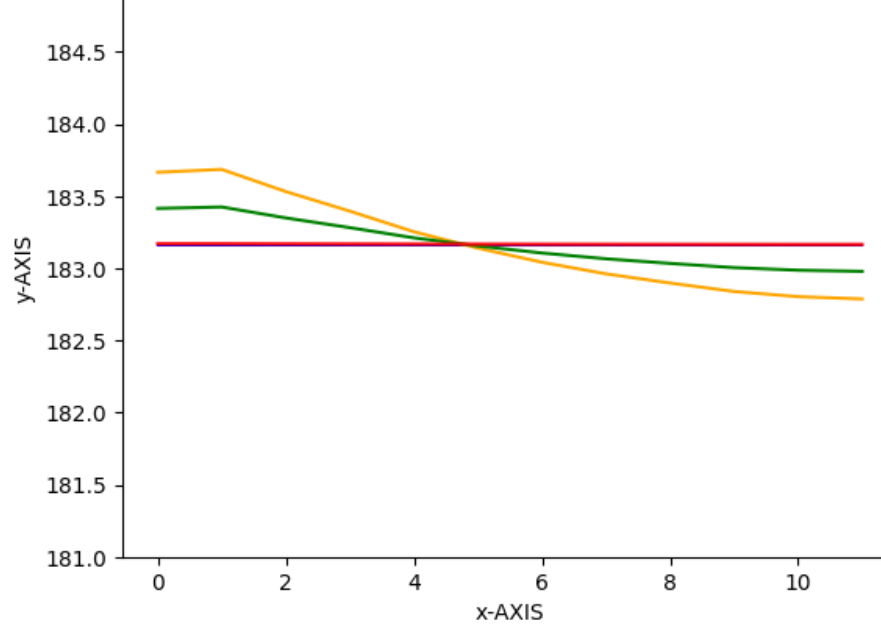}} \\ b)
\end{minipage}
\vfill
\begin{minipage}[h]{0.49\linewidth}
\center{\includegraphics[width=1\linewidth]{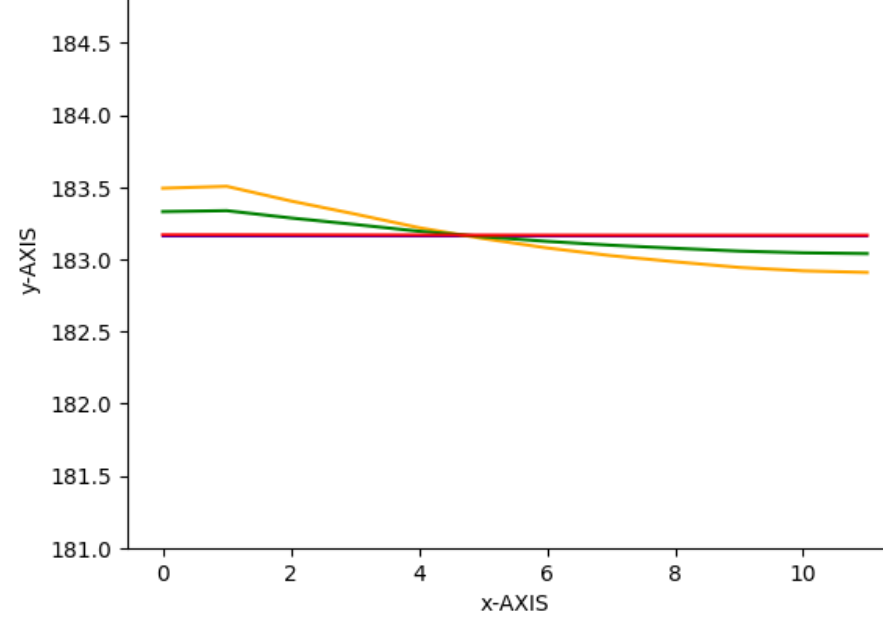}} \\ c)
\end{minipage}
\hfill
\begin{minipage}[h]{0.49\linewidth}
\center{\includegraphics[width=1\linewidth]{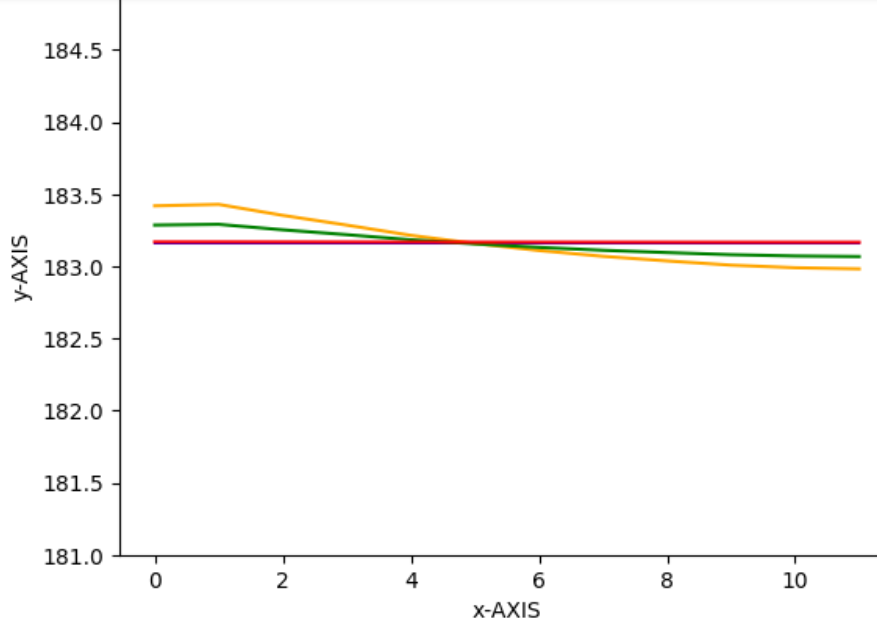}} \\ d)
\end{minipage}
\caption{The optimal broken line constructed by the proposed modification (orange), Ridge regression (green), Lasso regression (red) with $\overline{\alpha}$ values $\overline{\alpha}=10840$ (Fig. a), $\overline{\alpha}=21760$ (Fig. b), $\overline{\alpha}=32640$ (Fig. c), $\overline{\alpha}=43520$ (Fig. d) and the average line (blue) in Example \ref{AM_BA_exmpl4}}
\label{AM_BA_exmpl1_fig4}
\end{figure} 

Fig.\ref{AM_BA_exmpl1_fig4} shows that the LSM (least squares method) with proposed regularization converges more slowly to a straight line than Lasso and Ridge regression.

\end{example}

\section{Conclusion}
We proposed the new method of regularization for the problem of approximation of 2D set of points. The main idea behind the approach is the control of the length of the fitting curve. We demonstrated that the optimal curve is in the class of piecewise linear functions. We proved the theorem which provides the upper bound for the penalty coefficient for which the problem still makes sense, i.e. it can be guaranteed that the optimal solution differs from the trivial average line. Numerical examples demonstrate that results produced by our technique differs from the ones of Lasso and Ridge regularization.

\end{document}